\documentclass[a4paper,11pt,twoside]{amsart}
\usepackage[utf8]{inputenc}
\usepackage{amssymb,amsmath,amsthm,amscd}
\usepackage{mathrsfs}
\usepackage{bbm}
\usepackage{enumerate}
\usepackage{fullpage}
\usepackage{comment}

\usepackage{lmodern}

\newcommand{\C}{\ensuremath{\mathbb{C}}}

\newcommand{\cC}{\mathcal{C}}
\newcommand{\cH}{\mathcal{H}}

\newcommand{\cO}{\mathcal{O}}
\newcommand{\fg}{\mathfrak g}
\newcommand{\fh}{\mathfrak h}

\newcommand{\id}{\mathord{\text{\rm id}}}

\newcommand{\al}{\alpha}
\newcommand{\ot}{\otimes}

\newcommand{\uniC}{C_u(\mathcal C)}

\newcommand{\1}{\mathbbm{1}}
\newcommand{\tube}{\mathcal{AC}}
\newcommand{\cA}{\mathcal{A}}
\newcommand{\be}{\beta}
\newcommand{\ga}{\gamma}

\DeclareMathOperator{\Rep}{Rep}
\DeclareMathOperator{\Irr}{Irr}
\DeclareMathOperator{\Mor}{Mor}
\DeclareMathOperator{\End}{End}

\DeclareMathOperator{\mult}{mult}

\theoremstyle{definition}
\newtheorem{theorem}{Theorem}[section]
\newtheorem{theoremA}{Theorem}

\newtheorem*{corollaryB}{Corollary B}
\newtheorem{definition}[theorem]{Definition}
\newtheorem{lemma}[theorem]{Lemma}
\newtheorem{proposition}[theorem]{Proposition}
\newtheorem{corollary}[theorem]{Corollary}
\newtheorem{remark}[theorem]{Remark}

\author{Yuki Arano}
\thanks{YA was supported by a fellowship of the Japan Society for the Promotion of Science and the Program for Leading Graduate Schools, MEXT, Japan}
\address{Yuki Arano
\newline Department of Mathematics, Kyoto university
\newline Kitashirakawa-Oiwakecho, Sakyo-ku, Kyoto, 606-8502, Japan}
\email{y.arano@math.kyoto-u.ac.jp}

\author{Tim de Laat}
\thanks{TdL was supported by the Research Foundation -- Flanders (FWO) through a Postdoctoral Fellowship and by the Deutsche Forschungsgemeinschaft (SFB 878)}
\address{Tim de Laat
\newline Mathematisches Institut, Westf\"alische Wilhelms-Universit\"at M\"unster
\newline Einsteinstrasse 62, 48149, Germany}
\email{tim.delaat@uni-muenster.de}

\author{Jonas Wahl}
\thanks{JW was supported by European Research Council Consolidator Grant 614195}
\address{Jonas Wahl
\newline KU Leuven, Department of Mathematics
\newline Celestijnenlaan 200B -- Box 2400, B-3001 Leuven, Belgium}
\email{jonas.wahl@kuleuven.be}

\title{Howe-Moore type theorems for quantum groups and rigid $C^{\ast}$-tensor categories}

\begin{document}

\begin{abstract}
  We formulate and study Howe-Moore type properties in the setting of quantum groups and in the setting of rigid $C^{\ast}$-tensor categories. We say that a rigid $C^{\ast}$-tensor category $\mathcal{C}$ has the Howe-Moore property if every completely positive multiplier on $\mathcal{C}$ has a limit at infinity. We prove that the representation categories of $q$-deformations of connected compact simple Lie groups with trivial center satisfy the Howe-Moore property. As an immediate consequence, we deduce the Howe-Moore property for Temperley-Lieb-Jones standard invariants with principal graph $A_{\infty}$. These results form a special case of a more general result on the convergence of completely bounded multipliers on the aforementioned categories. This more general result also holds for the representation categories of the free orthogonal quantum groups and for the Kazhdan-Wenzl categories. Additionally, in the specific case of the quantum groups $\mathrm{SU}_q(N)$, we are able, using a result of the first-named author, to give an explicit characterization of the central states on the quantum coordinate algebra of $\mathrm{SU}_q(N)$, which coincide with the completely positive multipliers on the representation category of $\mathrm{SU}_q(N)$.
\end{abstract}

\maketitle

\section{Introduction} \label{sec:introduction}
A locally compact group has the Howe-Moore property if for every unitary representation without invariant vectors, the matrix coefficients of the representation vanish at infinity. This property was established for connected non-compact simple Lie groups with finite center by Howe and Moore \cite{howemoore} and Zimmer \cite{zimmer}. Howe and Moore also showed the property for certain subgroups of simple algebraic groups over non-Archimedean local fields. Other important examples of groups with the Howe-Moore property are given in \cite{lubotzkymozes} and \cite{burgermozes}. The Howe-Moore property plays an essential role in the proofs of several rigidity results, most notably in Margulis's superrigidity theorem and Mostow's rigidity theorem. More generally, the Howe-Moore property has important applications to ergodicity and strong mixingness of actions and flows. For an overview on and a unified approach to locally compact groups with the Howe-Moore property, we refer to \cite{ciobotaru} and \cite{badergelander}.

For Lie groups and algebraic groups over non-Archimedean local fields, we know much more than the Howe-Moore property about the asymptotic behaviour of matrix coefficients of representations. Recall that for every locally compact group $G$, the space $\mathrm{WAP}(G)$ of weakly almost periodic functions on $G$ admits a unique invariant mean. By a result of Veech \cite{veech} (see also \cite{ellisnerurkar}), it is known that every weakly almost periodic function on a connected non-compact simple Lie group with finite center has a limit at infinity, and this limit is equal to the mean of the function. As a consequence, it follows that a very strong version of the Howe-Moore property holds for such groups: for every uniformly bounded representation on a reflexive Banach space without invariant vectors, the matrix coefficients of the representation vanish at infinity. Recently, Bader and Gelander presented an abstract and general framework for such Howe-Moore type phenomena \cite{badergelander}, and in particular they explain that Veech's result also holds in the setting of connected simple algebraic groups over non-Archimedean local fields.\\

In this article, we initiate the study of Howe-Moore type phenomena in the setting of quantum groups and in the setting of rigid $C^{\ast}$-tensor categories. Unitary representations for quantum groups have been studied extensively, and recently, Popa and Vaes developed a unitary representation theory for ``subfactor related group-like objects'' in the setting of rigid $C^{\ast}$-tensor categories \cite{popavaes} (see also \cite{neshveyevyamashita}). Our main result is a Howe-Moore type theorem for the representation categories of $q$-deformations of compact simple Lie groups, which are ubiquitous and historically motivating examples of compact quantum groups. The theory of compact quantum groups was initiated by Woronowicz \cite{woronowicz}, and we refer to \cite{neshveyevtuset} for a thorough account of the theory. Recall that the representation category of a compact quantum group $\mathbb{G}$ is the category of finite-dimensional unitary representations of $\mathbb{G}$. It is a rigid $C^{\ast}$-tensor category, and, very strikingly, in the presence of a well-behaved functor into the category of finite-dimensional Hilbert spaces, there is a way to realize a rigid $C^{\ast}$-tensor category as the representation category of a certain compact quantum group (see e.g.~\cite{neshveyevtuset}).\\

Let $\mathcal{C}$ be a rigid $C^{\ast}$-tensor category, and let $\mathrm{Irr}(\mathcal{C})$ denote the set of equivalence classes of irreducible objects in $\mathcal{C}$. Then $\mathcal{C}$ is said to have the Howe-Moore property if every completely positive multiplier $\omega : {\rm{Irr}}({\mathcal{C}}) \to \mathbb{C}$ has a limit at infinity (see Definition \ref{dfn:hmtc}).
\begin{theoremA} \label{thm:hm}
	Let $q \in (0,1]$, and let $K_q$ be a $q$-deformation of a connected compact simple Lie group $K$ with trivial center. Then every completely bounded multiplier on $\Rep(K_q)$ has a limit at infinity. In particular, the representation category $\mathrm{Rep}(K_q)$ has the Howe-Moore property.
\end{theoremA}
Another rich source of rigid $C^*$-tensor categories is Jones's theory of subfactors. Every inclusion $N \subset M$ of II$_1$-factors with finite index $[M:N] < \infty$ can be enhanced to its Jones tower $M_{-1} \subset M_0 \subset M_1 \subset M_2 \subset \ldots$ of II$_1$-factors, where $M_{-1}=N$ and $M_0=M$, satisfying $[M_{i+1}: M_i] = [M:N]$ for all $i \geq 0$ and having the property that every $M_{i+1}$ is generated by $M_i$ and a projection $e_i$ commuting with $M_{i-1}$ (see \cite{jones1} for details). The projection $e_i$ is commonly referred to as the $i$-th Jones projection. The relative commutants $M_i^{\prime} \cap M_j$, with $i \leq j$, of algebras in the Jones tower form a lattice of finite-dimensional $C^{\ast}$-algebras. This lattice is called the standard invariant of $N \subset M$ and has played a fundamental role in many aspects of subfactor theory, such as the classification of subfactors of small index (see e.g.~\cite{jonesmorrisonsnyder}). Standard invariants have been abstractly characterized in several different ways, most notably as $\lambda$-lattices by Popa in \cite{popa2} and, more diagrammatically, as subfactor planar algebras by Jones \cite{jones2}. It is a striking result due to Popa (see \cite{popa2}) that every abstract $\lambda$-lattice (and hence every subfactor planar algebra) can be realized as a concrete one coming from a subfactor $N \subset M$ of index $[M:N] = \lambda^{-1}$.

A central example of a standard invariant is the Temperley-Lieb-Jones standard invariant $\rm{TLJ}(\lambda)$. It embeds into any other standard invariant coming from a subfactor $N \subset M$ with index $[M:N] = \lambda^{-1}$, so it can be viewed as an initial object for the category of standard invariants. With the Jones tower of $N \subset M$, we can naturally associate a rigid $C^{\ast}$-tensor category $\mathcal{C}_M$ consisting of all $M$-bimodules that are isomorphic to a finite direct sum of $M$-subbimodules of ${}_M L^2(M_i)_M$, where $i \geq 0$. The tensor operation, or fusion, in this category is the Connes tensor product over $M$. It is known that the category $\mathcal{C}_M$ is equivalent, as a rigid $C^{\ast}$-tensor category, to the representation category of the compact quantum group $\mathrm{PSU}_q(2)$, where $q$ is the unique number $0 < q \leq 1$ such that $q + \frac{1}{q} = \lambda^{-\frac{1}{2}}$ (see e.g.~\cite{popavaes}). The following result is therefore an immediate consequence of Theorem \ref{thm:hm}.
\begin{corollaryB}
Let $N \subset M$ be an inclusion of II$_1$ factors with index $[M:N] = \lambda^{-1} \geq 4$ and Temperley-Lieb-Jones standard invariant $\rm{TLJ}(\lambda)$ (and hence principal graph $A_{\infty}$). The rigid $C^{\ast}$-tensor category $\cC_M$ of $M$-bimodules associated with the Jones tower of $N \subset M$ has the Howe-Moore property.
\end{corollaryB}
Theorem \ref{thm:hm} (and hence Corollary B) follows from a more general result (see Theorem \ref{thm:generalhm}) on the convergence of completely bounded multipliers on certain rigid $C^{\ast}$-tensor categories. This more general result also holds for the representation categories of the free orthogonal quantum groups and for the Kazhdan-Wenzl categories. We refer to Section \ref{sec:thma} for the details.

Our Howe-Moore type results are the first such phenomena beyond the setting of locally compact groups. They may be useful for the study of quantum group actions. More specifically, the central Howe-Moore property (see Section \ref{sec:thma}) may be of use in the setting of actions of discrete quantum groups on certain operator algebras (cf.~\cite{dawsskalskiviselter}). More generally, one may apply the Howe-Moore property directly in the setting of actions of rigid $C^{\ast}$-tensor categories, which are just tensor functors from the tensor category to the category of bimodules over a von Neumann algebra (cf.~\cite{popa1}, \cite{hayashiyamagami}, \cite{masuda}).

For the proof of Theorem \ref{thm:generalhm} (and hence of Theorem \ref{thm:hm}), we relate the character algebra of $K_q$ to the character algebra of $K$, which can in turn be identified with the algebra of continuous functions on $T/W$, where $T$ is a maximal torus and $W$ is the associated Weyl group. A crucial ingredient of our proof is a certain general asymptotic behaviour of the characters of highest weight representations of compact Lie groups (see Proposition \ref{prop:characterconvergence}).\\

The Howe-Moore property for rigid $C^{\ast}$-tensor categories describes certain asymptotic behaviour of completely positive multipliers. The family of completely positive multipliers on the category $\mathrm{Rep}(\mathbb{G})$ (which is more restrictive than the family of completely bounded multipliers) coincides with the central states on the quantum coordinate algebra $\mathcal{O}(\mathbb{G})$, which were already investigated in \cite{decommerfreslonyamashita}. For the most well-known examples of $q$-deformations, namely the quantum groups $\mathrm{SU}_q(N)$, where $N \geq 2$, it turns out that we can find an explicit characterization of the central states. To this end, we use recent work of the first-named author \cite{arano2}, in which the unitary representation theory of the Drinfel'd double of a $q$-deformation of a compact Lie group is compared with the unitary representation theory of the complexification of the compact Lie group. Indeed, we are able to relate the central states very concretely to $\mathrm{SU}(N)$-bi-invariant positive-definite functions on $\mathrm{SL}(N,\mathbb{C})$. Below, let $Q$ denote the root lattice and $P$ the weight lattice with respect to a fixed Cartan subalgebra (see Section \ref{subsec:qdeformations}). Recall that a central state on $\mathcal{O}(\mathrm{SU}_q(N))$ is uniquely determined by its restriction to a chosen set $P_{+}$ of positive weights. Hence, we can identify the central states with maps $\varphi : P_{+} \to \mathbb{C}$.
\setcounter{theoremA}{2}
\begin{theoremA} \label{thm:cssuqn}
	For every central state $\varphi: P_+ \to {\mathbb C}$ on $\mathrm{SU}_q(N)$, there exist $\mathrm{SU}(N)$-bi-invariant positive-definite functions $\varphi_\chi$ on $\mathrm{SL}(N,\mathbb{C})$ such that
	\begin{align*}
	  \varphi(\lambda) &= \sum_{\chi \in \widehat{P/Q}} \chi(\lambda) \varphi_\chi(q^{2\lambda + 2\rho}), \; \textrm{and}\\
	  \| \varphi \| &\leq \sum_{\chi \in \widehat{P/Q}} \|\varphi_\chi \| \leq C \| \varphi \|,
	\end{align*}
	where $C$ is a constant (not depending on $\varphi$) and $\rho$ is half of the sum of the positive roots.
\end{theoremA}
This explicit characterization of central states can be useful for the study of the analytic properties of the quantum groups $\mathrm{SU}_q(N)$ that involve central states.

The article is organized as follows. We recall some preliminaries in Section \ref{sec:preliminaries}. Theorem \ref{thm:hm} is proven in Section \ref{sec:thma} and Theorem \ref{thm:cssuqn} in Section \ref{sec:thmb}.

\section*{Acknowledgements}
The authors thank Stefaan Vaes for valuable discussions, suggestions and remarks. The first-named author wishes to thank KU Leuven, where this research was carried out, for the invitation and hospitality.

\section{Preliminaries} \label{sec:preliminaries}

\subsection{Rigid $C^{\ast}$-tensor categories} \label{subsec:cstc}
A $C^{\ast}$-tensor category is a category that behaves similar to the category of Hilbert spaces. For the basic theory of $C^{\ast}$-tensor categories and the facts mentioned in this subsection, we refer to \cite[Chapter 2]{neshveyevtuset}.

In what follows, all tensor categories will be assumed to be strict, unless explicitly mentioned otherwise. This is not a fundamental restriction, since every tensor category can be strictified.

Let $\mathcal{C}$ be a $C^{\ast}$-tensor category. An object $\bar{u}$ in $\mathcal{C}$ is conjugate to an object $u$ in $\mathcal{C}$ if there are $R \in \Mor(\1,  \bar{u} \ot u)$ and $\bar{R} \in \Mor(\1, u \ot \bar{u})$ such that
\[
  u \xrightarrow{1 \ot R} u \ot \bar{u} \ot u \xrightarrow{\bar{R}^* \ot 1 } u \ \ \text{and} \ \ \bar{u} \xrightarrow{1 \ot \bar{R}} \bar{u} \ot u \ot \bar{u} \xrightarrow{R^* \ot 1} \bar{u}
\]
are the identity morphisms. Conjugate objects are uniquely determined up to isomorphism. If every object has a conjugate object, then the category $\mathcal{C}$ is called a rigid $C^{\ast}$-tensor category.

Let $\mathrm{Irr}(\mathcal{C})$ denote the set of equivalence classes of irreducible objects in $\mathcal{C}$. Using the same notation as above, if $u$ is an irreducible object with a conjugate, then $d(u) = \|R \| \| \bar{R} \|$ is independent of the choice of the morphisms $R$ and $\bar{R}$. An arbitrary object $u$ in a rigid $C^{\ast}$-tensor category is unitarily equivalent to a direct sum $u \cong \bigoplus_{k} u_k$ of irreducible objects, and we put $d(u) = \sum_{k} d(u_k)$. The function $d : \mathcal{C} \to [0, \infty)$ defined in this way is called the intrinsic dimension of $\mathcal{C}$.   

\subsection{Multipliers on rigid $C^{\ast}$-tensor categories}
Multipliers on rigid $C^{\ast}$-tensor categories were introduced by Popa and Vaes \cite{popavaes}.
\begin{definition} \label{mult}
A multiplier on a rigid $C^{\ast}$-tensor category $\cC$ is a family of linear maps 
\[
  \theta_{\alpha,\beta} : \End(\alpha \ot \beta) \to  \End(\alpha \ot \beta)
\]
indexed by $\alpha, \beta \in \cC$ such that
\begin{align}
  \theta_{\alpha_2,\beta_2}(UXV^{\ast}) &= U\theta_{\alpha_1,\beta_1}(X)V^{\ast}, \nonumber \\
  \theta_{\alpha_1 \ot \alpha_2,\beta_1 \ot \beta_2} (1 \ot X \ot 1) &= 1 \ot \theta_{\alpha_2, \beta_1}(X) \ot 1 \label{eq:equation1}
\end{align}
for all $\alpha_i, \beta_i \in \cC, X \in \End(\alpha_2 \ot \beta_1)$ and $U,V \in \mathrm{Mor}(\alpha_1,\alpha_2) \ot \mathrm{Mor}(\beta_2,\beta_1)$.
\end{definition}
A multiplier $(\theta_{\alpha,\beta})$ is said to be completely positive (or a cp-multiplier) if all maps $\theta_{\alpha,\beta}$ are completely positive. A multiplier $(\theta_{\alpha,\beta})$ is said to be completely bounded (or a cb-multiplier) if all maps $\theta_{\alpha,\beta}$ are completely bounded and $\| \theta \|_{\mathrm{cb}} = \sup_{\alpha, \beta \in \cC} \| \theta_{\alpha,\beta}\|_{\mathrm{cb}} < \infty$. By \cite[Proposition 3.6]{popavaes}, every multiplier corresponds uniquely to a function $\varphi: \Irr(\cC) \to \C$ and we will often mean such a function when we speak of a multiplier.

\subsection{The fusion algebra and admissible $\ast$-representations} \label{subsec:admissiblerepresentations}
Recall that the fusion algebra $\C[\cC]$ of a rigid $C^{\ast}$-tensor category $\cC$ is defined as the free vector space with basis $\Irr(\cC)$ and multiplication given by
\[
  \alpha \beta = \sum_{\gamma \in \Irr(\cC)} \mult(\alpha \ot \beta, \gamma) \gamma, \ \ \ \alpha, \beta \in \Irr(\cC).
\]
In fact, the fusion algebra is a $\ast$-algebra when equipped with the involution $\alpha^{\sharp}=\bar{\alpha}$.

In \cite{popavaes}, Popa and Vaes defined the notion of admissible $\ast$-representation of $\mathbb{C}[\mathcal{C}]$ as a unital $\ast$-representation $\Theta: \C[\cC] \to B(\cH)$ such that for all $\xi \in \cH$ the map
\[
  \Irr(\cC) \to \C, \; \alpha \to d(\alpha)^{-1} \langle \Theta(\alpha) \xi, \xi \rangle
\]
is a cp-multiplier. Moreover, they proved the existence of a universal admissible $\ast$-representation and denoted the corresponding enveloping $C^{\ast}$-algebra of $\C[\cC]$ by $\uniC$.

\subsection{The tube algebra}
In \cite{ghoshjones}, the representation theory of rigid $C^{\ast}$-tensor categories was related to Ocneanu's tube algebra, which was introduced in \cite{ocneanu}. Let us recall the definition of the tube algebra. Let $\cC$ be a rigid $C^{\ast}$-tensor category. For each equivalence class $\alpha \in \mathrm{Irr}(\mathcal{C})$, choose a representative $X_{\alpha} \in \alpha$, and let $X_0$ denote the representative of the tensor unit. Moreover, let $\Lambda$ be a countable family of equivalence classes of objects in $\cC$ with distinct representatives $Y_{\beta} \in \beta$ for every $\beta \in \Lambda$. The annular algebra with weight set $\Lambda$ is defined as
\[
  \cA \Lambda = \bigoplus_{\al, \be \in \Lambda, \ \ga \in \Irr(\cC)} \Mor(X_{\ga} \ot Y_{\al}, Y_{\be} \ot X_{\ga}).
\]
The algebra $\cA \Lambda$ comes equipped with the structure of an associative $*$-algebra. We will always assume the weight set $\Lambda$ to be full, i.e.~every irreducible object is equivalent to a subobject of some element in $\Lambda$. The annular algebra with weight set $\Lambda = \Irr(\cC)$ is called the tube algebra of Ocneanu, and we write $\cA \Lambda = \tube$.

\subsection{$q$-deformations of compact simple Lie groups} \label{subsec:qdeformations}
A compact quantum group is a pair $(A,\Delta)$ consisting of a unital $C^{\ast}$-algebra $A$ and a unital $\ast$-homomorphism $\Delta : A \to A \otimes A$ (comultiplication) such that $(\Delta \otimes \id)\Delta=(\id \otimes \Delta)\Delta$ and such that $\mathrm{span}\{(A \otimes 1)\Delta(A)\}$ and $\mathrm{span}\{(1 \otimes A)\Delta(A)\}$ are dense in $A \otimes A$. The tensor product $\otimes$ denotes the minimal tensor product. For a recent thorough introduction to the theory of compact quantum groups, we refer to \cite{neshveyevtuset}.

Another class of quantum groups is the class of discrete quantum groups, which is dual to the class of compact quantum groups under an appropriate generalization of Pontryagin duality. In fact, the search for such an appropriate notion of duality was a central motivation in the early days of quantum group theory. The most general analytic framework for quantum groups is the theory of locally compact quantum groups, as introduced by Kustermans and Vaes \cite{kustermansvaes}. It includes both the compact and discrete quantum groups and the locally compact groups, and it provides a satisfying answer to the duality question. This theory is, however, significantly more complicated.

Notable examples of compact quantum groups are the $q$-deformations of compact Lie groups. We recall their construction below. For details, see e.g.~\cite[Section 2.4]{neshveyevtuset} or \cite{klimykschmuedgen}. Let $K$ be a connected simply connected compact simple Lie group. We restrict ourselves to the case of $K$ being simple, because we need this later. However, for the construction of $q$-deformations, this is not essential. Let $G = K_\mathbb{C}$ be the complexification of $K$. The group $G$ has an Iwasawa decomposition $G=KAN$. Let $\fg$ be the Lie algebra of $G$ and fix a Cartan subsalgebra $\fh$ of $\fg$. Let $\Delta$ be the associated set of roots, let $Q \subset \fh^*$ be the root lattice and $P \subset \fh^*$ the weight lattice. Denote by $(\cdot,\cdot)$ the natural bilinear form on $\fh$, which we assume to be normalized by $(\alpha,\alpha)=2$ for a short root $\alpha$. For each $\alpha \in \Delta$, define the coroot as $\alpha^\vee = \frac{2\alpha}{(\alpha,\alpha)}$. Choose a set $\Pi=\{\alpha_i \mid i \in I\}$ of simple roots, and let $\Delta_+$ denote the set of positive roots, $Q_+$ the positive elements in the root lattice and $P_+$ the positive weights. Put $d_i = \frac{(\alpha_i,\alpha_i)}{2}$, and denote by $a_{ij}=\frac{(\alpha_i,\alpha_j)}{d_i}$ the entries of the Cartan matrix.

Fix $q \in (0,1)$. Define $q_i = q^{d_i}$, and set
\[
  [n]_q = \frac{q^n - q^{-n}}{q - q^{-1}}, \qquad [n]_q! = [n]_q [n-1]_q \dots [1]_q \qquad \textrm{ and } \qquad \left[ \begin{matrix} n \\ m \end{matrix} \right]_q = \frac{[n]_q!}{[m]_q! [n-m]_q!}.
\]
The quantized enveloping algebra $U_q(\fg)$ of $\mathfrak{g}$ is the unital algebra defined by the generators $\{K_i^{\pm 1}, E_i, F_i \mid i \in I\}$ and the relations
\begin{align*}
  K_i K_i^{-1} = K_i^{-1} K_i = 1, &\qquad K_i K_j = K_j K_i,\\
  K_i E_j K_i^{-1} = q^{a_{ij}}_iE_j, &\qquad K_i F_j K_i^{-1} = q^{-a_{ij}}_iF_j,
\end{align*}
\[
  [E_i,F_j] = \delta_{ij} \frac{K_i - K_i^{-1}}{q_i - q_i^{-1}},\\
\]
\[
  \sum_{r=0}^{1-a_{ij}} (-1)^r \left[ \begin{matrix} 1- a_{ij} \\ r \end{matrix} \right]_{q_i} E_i^r E_j E_i^{1-a_{ij}-r} = 0, \qquad \sum_{r=0}^{1-a_{ij}} (-1)^r \left[ \begin{matrix} 1-a_{ij} \\ r \end{matrix} \right]_{q_i} F_i^r F_j F_i^{1-a_{ij}-r} = 0.
\]
Note that the quantized enveloping algebra is a deformation of the universal enveloping algebra of $\mathfrak{g}$. In fact, the quantized enveloping algebra $U_q(\mathfrak{g})$ can be turned into a Hopf $*$-algebra by defining a comultiplication $\hat{\Delta}_q$ and involution $^{\ast}$ by
\begin{align*}
  &\hat{\Delta}_q(K_i)=K_i \otimes K_i, \qquad \hat{\Delta}_q(E_i)=E_i \otimes 1 + K_i \otimes E_i, \qquad \hat{\Delta}_q(F_i)=F_i \otimes K_i^{-1} + 1 \otimes F_i,\\
  &K_i^{\ast}=K_i, \qquad E_i^{\ast}=F_iK_i, \qquad F_i^{\ast}=K_i^{-1}E_i.
\end{align*}
Recall that the counit $\hat{\varepsilon}_q$ and the antipode $\hat{S}_q$ are given by the formulas
\[
  \hat{\varepsilon}_q(K_i)  =  1, \qquad \hat{\varepsilon}_q(E_i) = \hat{\varepsilon}_q (F_i) = 0,
\]
\[
  \hat{S}_q (K_i) = K_i^{-1}, \qquad \hat{S}_q(E_i) = -K_i^{-1} E_i, \qquad \hat{S}_q(F_i) = -F_i K_i.
\]

Let $V$ be an $U_q(\fg)$-module, and let $\mu \in P$. The weight space $V_{\mu}$ is defined as
\[
  V_{\mu} = \{ v \in V \mid K_i v = q_i^{(\mu, \al_i^{\vee}) } v \ \forall i \in I \}.
\]
The module $V$ is said to be of type $1$ if $V$ decomposes as a direct sum $V = \bigoplus_{\mu \in P}$. If a vector $v \in V$ is an element of the direct summand $V_{\mu}$, we sometimes refer to the weight $\mu$ as ${\rm wt}(v)$.

For each $\lambda \in P_+$, there exists a uniquely determined irreducible module $V(\lambda)$ of highest weight $\lambda$, i.e.~$V(\lambda) = U_q(\fg) v_{\lambda}$ for a nonzero vector $v_\lambda \in V(\lambda)$ satisfying $K_i v_\lambda = q^{(\lambda,\alpha_i^\vee)}_i v_\lambda$ and $E_i v_\lambda = 0$ for all $i \in I$. The module $V(\lambda)$ is finite-dimensional and admits an invariant inner product.

For $v,w \in V_{\mu}$, define $u^\mu_{vw} \in U_q(\fg)^*$ by $(u^\mu_{vw},x) = (xv,w)$. Let $\cO(K_q)$ be the quantum coordinate algebra, i.e.~the subspace of $U_q(\fg)^*$ consisting of matrix coefficients of finite-dimensional unitary modules of type 1. More precisely,
\[
  \cO(K_q) = {\rm span}\{u^\mu_{vw} \mid \mu \in P, v,w \in V_{\mu}\}.
\]
The algebra $\cO(K_q)$ admits a unique Hopf $*$-algebra structure that turns the pairing $\cO(K_q) \times U_q(\fg) \to {\mathbb C}$ into a Hopf $\ast$-algebra pairing. More concretely, multiplication, involution and comultiplication on $\cO(K_q)$ are given by the formulas
\[
  (ab)(x) = (a\ot b)(\hat{\Delta}_q(x)), \qquad a^*(x) = \overline{a(\hat{S}_q(x)^*)}, \qquad \textrm{and} \qquad \Delta_q(a)(x \ot y) = a(xy)
\]
for $a,b \in \cO(K_q)$ and $x,y \in U_q(\fg)$. The universal $C^{\ast}$-completion $C(K_q)$ of the $\ast$-algebra $\cO(K_q)$ is (the $C^{\ast}$-algebra associated with) the compact quantum group $K_q$.

\section{The Howe-Moore property for representation categories} \label{sec:thma}
In this section, we prove the Howe-Moore property for the representation categories of $q$-deformations of connected compact simple Lie groups with trivial center, i.e.~Theorem \ref{thm:hm}.
\begin{definition} \label{dfn:hmtc}
  A rigid $C^{\ast}$-tensor category $\mathcal{C}$ is said to have the Howe-Moore property if for every completely positive multiplier $\omega : {\rm{Irr}}({\mathcal{C}}) \to \mathbb{C}$, we have $\omega \in c_0(\mathrm{Irr}(\mathcal{C})) \oplus \mathbb{C}$.
\end{definition}
In the case where $\mathcal{C}$ is the representation category of a compact quantum group $\mathbb{G}$, it turns out that the Howe-Moore property for $\mathrm{Rep}(\mathbb{G})$ is equivalent to a central version of the Howe-Moore property, i.e.~a version of the Howe-Moore property for the central states on the quantum coordinate algebra $\mathcal{O}(\mathbb{G})$ of $\mathbb{G}$. This central Howe-Moore property can be viewed as a property for the dual $\widehat{\mathbb{G}}$ of $\mathbb{G}$. The dual of a compact quantum group is a discrete quantum group. As indicated in Section \ref{sec:introduction}, the approach of structural properties of quantum groups through central versions of these properties for their duals goes back to \cite{decommerfreslonyamashita}. It also played an important role in the work of the first-named author on property (T) in the setting of quantum groups \cite{arano1}.
\begin{definition} \label{dfn:chm}
  Let $\mathbb{G}$ be a compact quantum group. The discrete quantum group $\widehat{\mathbb{G}}$ is said to have the central Howe-Moore property if $\omega \in c_0(\widehat{\mathbb{G}}) \oplus \mathbb{C}$ for every central state $\omega$ on $\mathcal{O}(\mathbb{G})$.
\end{definition}
The following result relates the Howe-Moore property for representation categories to the central Howe-Moore property of the duals of the underlying quantum groups. It follows immediately from the fact that completely positive multipliers on the category $\mathrm{Rep}(\mathbb{G})$ coincide with the central states of $\mathbb{G}$ (see \cite[Proposition 6.1]{popavaes}).
\begin{proposition}
  Let $\mathbb{G}$ be a compact quantum group. The representation category $\mathrm{Rep}(\mathbb{G})$ has the Howe-Moore property for rigid $C^{\ast}$-tensor categories if and only if the dual $\widehat{\mathbb{G}}$ of $\mathbb{G}$ has the central Howe-Moore property.
\end{proposition}
Let us now start working towards the proof of Theorem \ref{thm:hm}.
\begin{lemma} \label{lemma:subrootsystem}
	Let $\Delta$ be a root system with a root subsystem $\Delta^0 \subset \Delta$. Then for all $\alpha \in \Delta$, we have the following: whenever there exists $\beta \in \Delta \setminus \Delta^0$ such that $(\alpha,\beta) \neq 0$, then $\alpha \in \mathop{\rm span} (\Delta \setminus \Delta^0)$.
	
	In particular, if $\Delta$ is irreducible and $\Delta^0$ is a proper subsystem, then $\mathop{\rm span} (\Delta \setminus \Delta^0) = \fh^*$.
\end{lemma}
\begin{proof}
The first assertion follows in a straightforward way from a consideration of root systems with rank $2$. The second assertion follows from the fact that $\Delta^1 = \Delta \cap \mathop{\rm span} (\Delta \setminus \Delta^0)$ and $\Delta^2 = \Delta \setminus \Delta^1$ are root systems that are perpendicular to each other.
\end{proof}
The following result constitutes a crucial ingredient of our approach. Let $K$ be a connected compact simple Lie group (with center $Z(K)$), and fix a Cartan subalgebra of its Lie algebra. Let $T$ be the associated maximal torus, $\Delta$ the set of roots (with positive part $\Delta_+$) and $P$ the weight lattice (with positive part $P_+$). Let $\rho = \frac{1}{2} \sum_{\alpha \in \Delta_+} \alpha$. For $\lambda \in P_+$, the highest weight representation is denoted by $V(\lambda)$. The character of $V(\lambda)$ is denoted by $\chi_{\lambda}$. We refer to Section \ref{subsec:qdeformations} for details on the aforementioned structures.
	\begin{proposition} \label{prop:characterconvergence}
	Let $K$ be a connected compact simple Lie group, and let $T$, $P_+$ and $V(\lambda)$ be as above. For every $t \in T \setminus Z(K)$, we have
	\[
	  \frac{1}{\dim(V(\lambda))}\chi_\lambda(t) \to 0 \qquad \text{as} \qquad \lambda \to \infty.
	\]
	\end{proposition}	
	\begin{proof}
	We use the Weyl character formula (see e.g.~\cite[Theorem V.1.7]{helgason}), which computes the character $\chi_\lambda$ of a highest weight representation $V(\lambda)$ of highest weight $\lambda$ of a connected compact simple Lie group:
	\[
	  \chi_\lambda(e^H) = \frac{\sum_{w \in W} (-1)^{l(w)} e^{w(\lambda + \rho)(H)}}{e^{\rho(H)} \prod_{\alpha \in \Delta_+} (1-e^{-\alpha(H)})}
	\]
	and the Weyl dimension formula (see e.g.~\cite[Theorem V.1.8]{helgason}), which computes the dimension of the highest weight representation $V(\lambda)$:
	\[
	  \dim(V(\lambda)) = \frac{\prod_{\alpha \in \Delta_+} (\lambda + \rho,\alpha)}{\prod_{\alpha \in \Delta_+} (\rho, \alpha)}.
	\]
	In the rest of this section, we will write $t=e^H$ and e.g.~$t^{\alpha}$ for $e^{\alpha(H)}$.
	
	Fix $t_0 \in T \setminus Z(K)$. Then $\Delta^0 = \{\alpha \in \Delta \mid t_0^{-\alpha} = 1\} \neq \Delta$ is a root system with positive roots $\Delta_+^0 = \{\alpha \in \Delta_+ \mid t_0^{-\alpha} = 1\}$ and Weyl group $W_0 = \{w \in W \mid w t_0 = t_0\}$ (see \cite[Proposition 6.6]{kac}). Put $\rho_0 = \frac{1}{2} \sum_{\alpha \in \Delta^0_+} \alpha$. Fix representatives of $W_0 \backslash W$ in $W$.
	Then
	\begin{align*}
	\chi_\lambda(t_0) &= \lim_{t \to t_0} \sum_{w' \in W_0 \backslash W}  (-1)^{l(w')} \frac{1}{t^\rho \prod_{\alpha \in \Delta_+ \setminus \Delta_+^0}(1-t^{-\alpha})} \frac{\sum_{w \in W_0} (-1)^{l(w)}t^{w w'(\lambda + \rho)}} {\prod_{\alpha \in \Delta_+^0}(1-t^{-\alpha})} \\
	&= \lim_{s \to 1} \sum_{w' \in W_0 \backslash W}  (-1)^{l(w')} \frac{t_0^{w^{\prime}(\lambda + \rho)}}{t_0^{\rho}s^{\rho-\rho_0} \prod_{\alpha \in \Delta_+ \setminus \Delta_+^0}(1-(t_0 s)^{-\alpha})} \frac{\sum_{w \in W_0} (-1)^{l(w)}s^{w w'(\lambda + \rho)}} {s^{\rho_0}\prod_{\alpha \in \Delta_+^0}(1-s^{-\alpha})},
	\end{align*}
	where we have used the invariance of $t_0$ under the action of $W_0$ in the second equality. Note that by taking the limit $s \to 1$ in the Weyl character formula for the subsystem $\Delta^0$ in the same way as in the proof of the Weyl dimension formula, we obtain
	\[ \lim_{s \to 1} \frac{\sum_{w \in W_0} (-1)^{l(w)}s^{w w'(\lambda + \rho)}} {s^{\rho_0}\prod_{\alpha \in \Delta_+^0}(1-s^{-\alpha})}  \ = \ \frac{\prod_{\alpha \in \Delta_+^0} (w' (\lambda + \rho),\alpha)}{\prod_{\alpha \in \Delta_+^0} (\rho_0, \alpha)}. \]
	Moreover, since $t_0^{\rho} s^{\rho-\rho_0} \prod_{\alpha \in \Delta_+ \setminus \Delta_+^0}(1-(t_0 s)^{-\alpha})$ is non-zero whenever $s$ is sufficiently close to $1$, it suffices to show that for every $w'$, we have 
	\[
	  \frac{\prod_{\alpha \in \Delta_+^0} (w' (\lambda + \rho),\alpha)}{\prod_{\alpha \in \Delta_+} (\lambda + \rho,\alpha)} \to 0 \qquad \text{as} \qquad \lambda \to \infty.
	\]
	To this end, put
	\[\Delta_+^{w',0} = \{\beta \in \Delta_+ \mid \pm \beta \in (w')^{-1} \Delta_+^0\}.\]
	Then
	\[ \frac{\prod_{\alpha \in \Delta_+^0} (w' (\lambda + \rho),\alpha)}{\prod_{\alpha \in \Delta_+} (\lambda + \rho,\alpha)} = \pm \prod_{\alpha \in \Delta_+ \setminus \Delta_+^{w',0}}\frac{1}{(\lambda+\rho,\alpha)}.\]
	Note that every factor $\frac{1}{(\lambda+\rho,\alpha)}$ is at most $1$. Write $\lambda$ as a linear combination of $\varpi_i$, where $\varpi_i$ is the fundamental weight, i.e.\ $(\varpi_i,\alpha_j^\vee) = \delta_{ij}$. Then the maximum of the coefficients tends to infinity as $\lambda$ tends to infinity. Hence, we only need to show that for all $i \in I$, there exists an $\alpha \in \Delta_+ \setminus \Delta_+^{w',0}$ such that $(\varpi_i,\alpha) \neq 0$.
	
	Suppose that this is not the case. Then there exists an $i \in I$ such that for all $\alpha \in \Delta_+ \setminus \Delta_+^{w',0}$, we have $(\varpi_i,\alpha) = 0$. This shows that $(w' \varpi_i,\alpha)=0$ for any $\alpha \in \Delta \setminus \Delta^0$. From Lemma \ref{lemma:subrootsystem} and the assumption that $K$ is simple, we obtain that $(w' \varpi_i,\alpha) = 0$ for all $\alpha \in {\rm span} (\Delta_+ \setminus \Delta^0_+) = \fh^*$, which is a contradiction.
	\end{proof}
	Recall that the character algebra of $C(K_q)$ is the $C^{\ast}$-subalgebra of $C(K_q)$ spanned by
	\[
	  \chi^q_\lambda := \sum_{i} u^\lambda_{vv},
	\]
	where $u^\lambda_{vv}$ is as in Section \ref{subsec:qdeformations}.
	\begin{theorem} \label{thm:generalhm}
	Let $K$ be a connected compact simple Lie group, and let $\mathcal{C}$ be a rigid $C^{\ast}$-tensor category satisfying the fusion rules of $K$ (i.e.~the fusion ring of $\mathcal{C}$ is isomorphic to the fusion ring of $K$). For every completely bounded multiplier $\omega$ on $\mathcal{C}$, there exists a map $c: Z(K) \to \mathbb{C}$ such that
	\[
	  \omega(\lambda) - \sum_{t \in Z(K)} c(t) t^\lambda \to 0 \qquad \textrm{as} \qquad \lambda \to \infty.
	\]	
	In particular, if $K$ has trivial center, then the category $\mathcal{C}$ has the Howe-Moore property.
	\end{theorem}
	\begin{proof}
	Let $\omega: P_+ \to \mathbb{C}$ be a completely bounded multiplier on $\mathcal{C}$. Then $\omega$ gives rise to a normal completely bounded map $T : \tube'' \to \tube''$, where $\tube$ is the tube algebra. We restrict $T$ to the fusion algebra in order to get a completely bounded map
	\[
	  T \colon C(T/W) \to C(T/W), \chi_\lambda \mapsto \omega(\lambda) \chi_\lambda.
	\]
	After composing this with the evaluation map at the neutral element in $T$, we obtain a bounded functional
	\[
	  \tilde{\omega} \colon C(T/W) \to \mathbb{C} \colon \chi_\lambda \mapsto \dim (V(\lambda))  \omega(\lambda).
	\]
	Since $\tilde{\omega}$ is a bounded functional on $C(T/W)$, there exists a finite measure $\mu$ on $T/W$ such that
	\[
	  \omega(\lambda) = \frac{1}{\dim(V(\lambda))} \int_{t \in T/W} \chi_\lambda(t) d\mu(t).
	\]
	Put $c(t) = \mu(\{t\})$. By Proposition \ref{prop:characterconvergence} and the dominated convergence theorem, we obtain that
	\[
	  \omega(\lambda) - \sum_{t \in Z(K)} c(t) t^\lambda \to 0 \qquad \textrm{as} \qquad \lambda \to \infty.
	\]
	\end{proof}
Obvious examples of categories $\mathcal{C}$ satisfying the conditions of Theorem \ref{thm:generalhm} are the representation categories of $q$-deformations of connected compact simple Lie groups, where $q \in (0,1]$. If, moreover, the Lie group has trivial center, then the result directly implies Theorem \ref{thm:hm}.
\begin{remark}
It follows directly from the proof that Theorem \ref{thm:generalhm} also holds for $\mathrm{Rep}(\mathrm{SU}_q(2))$ for $q \in (-1,0)$. Since every free orthogonal quantum group $O_F^+$ is monoidally equivalent to $\mathrm{SU}_q(2)$ (and hence has the same representation category) for some uniquely determined $q \in [-1,0) \cup (0,1]$, the representation categories of free orthogonal quantum groups also satisfy the conditions of Theorem \ref{thm:generalhm}. Other examples of rigid $C^{\ast}$-tensor categories with the same fusion ring as some connected compact simple Lie group are the Kazhdan-Wenzl categories \cite{kazhdanwenzl} (see also \cite{jordans}).
\end{remark}

\section{A characterization of central states on $\mathrm{SU}_q(N)$} \label{sec:thmb}
Let $K$ be a connected simply connected compact simple Lie group, and take its complexification $G = K_\mathbb{C}$, which has Iwasawa decomposition $G=KAN$. Let $\mathfrak{g}$ be the Lie algebra of $G$, and let $\mathfrak{h}$ be a Cartan subalgebra of $\mathfrak{g}$. Let $K_q$ be a $q$-deformation of $K$, where $q \in (0,1)$. Recall the following definition from \cite{arano2}.
\begin{definition}
  We say that $\nu \in \fh^*$ is almost real (with respect to $q$) if $({\rm Im}(\nu),\alpha) < 2\pi \log(q)^{-1}$ for all $\alpha \in \Delta$.
\end{definition}
Recall the classification of extremal positive-definite functions of $G$. The spherical admissible dual of $G$ is homeomorphic to $\fh^*/W$ and the Berezin--Harish--Chandra formula (see \cite[Theorem 5.7]{helgason})  (with induction parameter $\frac{1}{2}\nu$) asserts that the corresponding $K$-bi-invariant function for $\nu \in \fh^*$ is
\[
  \varphi^\nu_1(q^\mu) = \frac{\chi_{\frac{1}{2}\nu-\rho}(q^\mu)}{\chi_{\frac{1}{2}\nu-\rho}(1)},
\]
where $\mu \in \fh^*_\mathbb{R}$, the element $q^\mu$ is an element in $A$, the weight $\rho$ is half of the sum of the positive roots (which equals the sum of the fundamental weights), and $\chi_\nu$ is the analytic continuation of the Weyl character formula:
\[
  \chi_\nu(q^\mu) = \frac{A_{\nu+\rho}(q^\mu)}{A_\rho(q^\mu)}, \quad \textrm{where} \quad A_\nu(q^\mu) = \sum_{w \in W} (-1)^{l(w)} q^{(\mu,w \nu)}.
\]
Note that this formula holds for every (non-unitary) spherical principal series representation, since both sides of the equation are analytic. The unitary spherical dual (equivalently, the set of zonal spherical functions) is parametrized by $\nu \in \fh^*$ such that $\varphi^\nu_1$ is positive-definite.

Recall that a central state $\omega$ on $C(K_q)$ is called extremal if $\omega$ is a extremal point in the set of the central states. The following theorem follows from \cite{arano2}.
\begin{theorem}\label{thm:arano}
For $q \in (0,1]$, we have the following:
\begin{enumerate}[(i)]
	\item the set of extremal central states on $K_q$ is parametrized by $\nu \in \fh^*/2\pi i \log(q)^{-1} Q^\vee \rtimes W$ such that
	\[
	  \varphi^\nu_q(\lambda) = \frac{{\chi}_\lambda(q^{\nu})}{{\chi}_\lambda(q^{2\rho})}
	\]
	is positive-definite;
	\item if $\nu \in \fh^*$ is almost real, then $\varphi^\nu_q$ is positive-definite if and only if $\varphi^\nu_1$ is;
	\item if $K=SU(N)$, then for all $\nu \in \fh^*$, we have $\chi \in 2 \pi i \log(q)^{-1} P^\vee$ such that $\nu - \chi$ is almost real with respect to $q$.
\end{enumerate}
\end{theorem}
\begin{proof}
Recall that by \cite[Theorem 29]{decommerfreslonyamashita}, the central states are in one-to-one correspondence with the $K_q$-invariant states on the quantum codouble $C_0^u(\hat{\mathcal{D}}K_q)$, and hence unitary representations with a $K_q$-invariant vector. Then we obtain (i) from \cite[Corollary 3.6]{arano1}, (ii) from \cite[Theorem 4.9]{arano2} and (iii) from \cite[Lemma 2.5]{arano2}.
\end{proof}
\begin{proposition} For $q \in (0,1]$, we have
  \[
    \varphi^\nu_q(\lambda) = \frac{\varphi^\nu_1(q^{2\lambda+2\rho})}{\varphi^\nu_1(q^{2\rho})}.
  \]
\end{proposition}
\begin{proof}
From the Weyl character formula, it follows that
\[
  \varphi^\nu_q(\lambda) = \frac{A_{\lambda+\rho}(q^\nu)}{A_\rho(q^\nu)} \frac{A_{\rho}(q^{2\rho})}{A_{\lambda+\rho}(q^{2\rho})}.
\]
Using that $A_\nu(q^\mu) = A_\mu(q^\nu)$, we compute that
\[
  \varphi^\nu_q(\lambda) = \frac{A_{\frac{1}{2}\nu}(q^{2\lambda+2\rho})}{A_{\rho}(q^{2\lambda+2\rho})} \frac{A_{\rho}(q^{2\rho})}{A_{\frac{1}{2}\nu}(q^{2\rho})} =  \frac{\varphi^\nu_1(q^{2\lambda+2\rho})}{\varphi^\nu_1(q^{2\rho})}.
\]
\end{proof}
Just as the Weyl character formula, the above formula should be viewed as a ``formal'' formula, which literally holds only for generic $\nu$. (For special $\nu$, we might get $0/0$.) Both the numerator and the denominator are analytic functions in the variable $\nu$, and we take the analytic extension of the left hand side in general. Thus we can compute the zeroes of the denominator.
\begin{proposition} \label{prp:zeroes}
  For all $\nu \in \fh^*$, we have $\varphi^\nu_1(q^{2 \rho}) = 0$ if and only if there exists $\alpha \in \Delta_+$ such that $(\nu,\alpha^\vee) \in 2 \pi i \log(q_\alpha)^{-1}  \mathbb{Z} \setminus \{0\}$.
\end{proposition}
\begin{proof}
Recall that
\[
  \varphi^\nu_1(q^{2 \rho}) = \frac{\chi_{\frac{1}{2}\nu - \rho}(q^{2\rho})}{\chi_{\frac{1}{2}\nu-\rho}(1)} = \frac{A_{\frac{1}{2}\nu}(q^{2\rho})}{A_\rho(q^{2\rho})} \frac{\prod_{\alpha \in \Delta_+}(\rho,\alpha)}{\prod_{\alpha \in \Delta_+} (\frac{1}{2}\nu,\alpha)},
\]
where we have used the Weyl dimension formula for the computation of the denominator. It follws that
\begin{itemize}
  \item the scalars $\prod_{\alpha \in \Delta_+}(\rho,\alpha)$ and $A_\rho(q^{2\rho})$ are nonzero;
  \item the function $\prod_{\alpha \in \Delta_+} (\frac{1}{2}\nu,\alpha)$ has a zero whenever $(\nu,\alpha) = 0$, and each zero has order $1$;
  \item and from the Weyl denominator formula, we obtain
    \[
      A_{\frac{1}{2}\nu}(q^{2\rho}) = A_\rho(q^\nu) = q^{(\rho,\nu)} \prod_{\alpha \in \Delta_+} (1 - q^{-(\alpha,\nu)} ),
    \]
    which has a zero whenever $(\nu,\alpha) \in 2 \pi i \log(q)^{-1} \mathbb{Z}$, and each zero has order $1$.
\end{itemize}
Combining these, we get the desired conclusion.
\end{proof}
\begin{corollary}\label{cor:positive}
  Let $\nu \in \fh^*$ be an almost real weight such that $\varphi^\nu_1$ is self-adjoint. Then we have $\varphi^\nu_1(q^{2\rho}) > 0$.
\end{corollary}
\begin{proof}
Since $q^{2\rho}$ and $q^{-2\rho}$ lie in the same $K$-orbit, we know that $\varphi^\nu_1(q^{2\rho})$ is real. We only need to show that the set $X$ of almost real weights $\nu$ such that $\varphi^\nu_1$ is self-adjoint is connected. Once one shows this, since $\varphi^\nu_1(q^{2\rho})$ is nonzero, we only need to check the positivity at an arbitrary point in $X$, say $\nu = 2\rho$. In this case $\varphi^{2 \rho}_1(q^{2\rho}) = 1$, which concludes the proof.

To show the claim, we may assume that ${\rm Re}(\nu)$ is dominant, that is to say, ${\rm Re}(\nu)$ lies in the closure of the fundamental Weyl chamber. Note that $\varphi^\nu_1$ is self-adjoint if and only if there exists $w \in W$ such that $\nu= - w \overline{\nu}$. Hence $\nu_t = {\rm Re}(\nu) + it {\rm Im}(\nu)$ is a homotopy in $X$ that connects $\nu$ and $\rm{Re}(\nu)$. On the other hand, the equation $\nu= - w \overline{\nu}$ shows that ${\rm Re}(\nu) = -w_0 {\rm Re}(\nu)$, where $w_0$ is the longest element of $W$. So the straight line connecting ${\rm Re}(\nu)$ and $2\rho$ is a homotopy in $X$.
\end{proof}
For $0<q\leq 1$, let $Z_q$ be the unitary spherical dual of $G_q$ inside $\fh^*/2\pi i\log(q)^{-1} Q^\vee \rtimes W$. We make the identification $2\pi i \log(q)^{-1} (P^\vee / Q^\vee) \simeq \widehat{P/Q}$.
\begin{lemma}\label{lem:decomposition}
Let $K=SU(N)$. Then we have a decomposition
\[
  Z_q = \coprod_{\chi \in \widehat{P/Q}} Z^\chi_q
\]
with the following properties:
\begin{itemize}
  \item if $\nu \in Z_q^\chi$, then $\nu-\chi \in Z_1$,
  \item the function $\nu \mapsto \frac{1}{\varphi^1_{\nu-\chi}(q^{2\rho})}$ is positive and bounded on $Z_q^\chi$.
\end{itemize}
\end{lemma}
\begin{proof}
  Immediate from Theorem \ref{thm:arano} and Corollary \ref{cor:positive}.
\end{proof}
We can now give the proof of Theorem \ref{thm:cssuqn}.
\begin{proof}[Proof of Theorem \ref{thm:cssuqn}]
Take the decomposition as in Lemma \ref{lem:decomposition}. Put
\[
  C = \sup\left\{\frac{1}{\varphi^1_{\nu-\chi}(q^{2\rho})} \,\Bigg\vert\, \nu \in Z_q^\chi, \chi \in \widehat{P/Q}\right\}.
\]
We know that $\varphi$ is positive-definite if and only if there exists a finite measure $\mu$ on the unitary spherical dual such that
\[
  \varphi = \int \varphi^\nu_q d\mu(\nu).
\]
By putting $d\mu_\chi(\nu) = d\mu(\nu+\chi)|_{Z^\chi_q}$, we obtain a measure on $Z_1$. Defining $\varphi_{\chi}$ as
\[
  \varphi_{\chi} = \int \varphi^\nu_1\frac{d\mu_{\chi}(\nu)}{\varphi^\nu_1(q^{2\rho})},
\]
this gives the desired decomposition.
\end{proof}

\end{document}